\documentclass[preprint,10pt]{scrartcl}

\usepackage[utf8]{inputenc}
\usepackage{amsmath}
\usepackage{amssymb}
\usepackage{amsthm}
\usepackage{enumerate}
\usepackage[english]{babel}
\usepackage{url}

\newcommand{\PG}{\text{PG}}

\newcommand{\tH}{{\bf H}}
\newcommand{\tQ}{{\bf Q}}
\newcommand{\tI}{{\bf I}}
\newcommand{\tQp}{{\bf Q^+}}

\newcommand{\I}{\text{ \tI }~}

\newcommand{\scrF}{\mathcal{F}}
\newcommand{\scrG}{\mathcal{G}}

\newcommand{\scrL}{\mathcal{L}}

\newcommand{\scrP}{\mathcal{P}}

\newcommand{\scrS}{\mathcal{S}}

\numberwithin{equation}{section}

\theoremstyle{plain}
\newtheorem{satz}[equation]{Theorem}

\newtheorem{lemma}[equation]{Lemma}
\newtheorem{prop}[equation]{Proposition}

\newtheorem{kor}[equation]{Corollary}

\theoremstyle{definition}

\theoremstyle{remark}

\begin{document}

\title{A Characterization of the Natural Embedding of the Split Cayley Hexagon in $\PG(6,q)$ by Intersection Numbers in Finite Projective Spaces of Arbitrary Dimension}
\author{Ferdinand Ihringer}
\maketitle

\begin{abstract}
  We prove that a non-empty set $\scrL$ of at most $q^5+q^4+q^3+q^2+q+1$ lines of $\PG(n, q)$ with the properties that (1) every point of $\PG(n,q)$ is incident with either $0$ or $q+1$ elements of $\scrL$, (2) every plane of $\PG(n, q)$ is incident with either $0$, $1$ or $q+1$ elements of $\scrL$, (3) every solid of $\PG(n, q)$ is incident with either $0$, $1$, $q+1$ or $2q+1$ elements of $\scrL$, and (4) every $4$-dimensional subspace of $\PG(n, q)$ is incident with at most $q^3-q^2+4q$ elements of $\scrL$, is necessarily the set of lines of a split Cayley hexagon $\tH(q)$ naturally embedded in $\PG(6, q)$.
\end{abstract}

\section{Introduction}

The characterization of embeddings of geometries in other geometries is a traditional topic of finite geometry.
It is quite common to find particular structures embedded in disguise.
Alternative descriptions of these objects make it easier to recognize them in different settings.

Embeddings of the split Cayley hexagon $\tH(q)$ in projective spaces were investigated intensively in recent decades. 
Besides the standard embedding of $\tH(q)$ in $\PG(6, q)$, and in $\PG(5, q)$ for $q$ even, (both introduced in \cite{MR1557095}) many alternative descriptions of $\tH(q)$ are known.
For example, there exists a description of $\tH(q)$ embedded in $\PG(3, q)$ (see \cite[Theorem 1.1]{MR2964273}, first described in \cite{MR549937} in the construction following Theorem A.7), and a description of an embedding in $H(3, q^2)$ (see \cite[Theorem 1.2]{MR2964273}).

It is a typical method to characterize an object in a finite projective space by its intersection numbers with the subspaces of a projective space. Popular or recent examples are quadrics (see \cite[ch. 22.10 and ch. 22.11]{hirschfeld1991general}), polar spaces in general (see \cite{MR2663559}), Veroneseans (see \cite{MR2927612}), Hermitian Veroneseans (see \cite{MR2128336}), and Segre Varieties (see \cite{MR2860661}).

Several descriptions of the standard embedding of the split Cayley hexagon by intersection numbers are known. The author is aware of a characterization of the line set of the natural symplectic embedding of the split Cayley hexagon in $\PG(5, q)$, $q$ even (see \cite{MR2515274}), a characterization of the line set of the standard embedding of the split Cayley hexagon in the parabolic quadric $\tQ(6, q)$ (see \cite[Theorem 1.1]{MR1633172}), and the inspiration for this work: a characterization of the standard embedding of the line set of the split Cayley hexagon in the projective space $\PG(6, q)$ (see \cite{maldeghem_splitcaley_by_intersect}).

The specific motivation for characterizing the standard embedding of the split Cayley hexagon by intersection numbers of lines is stated in length in \cite{maldeghem_splitcaley_by_intersect}.
Recent research in locally $d$-dimensional embeddings motivates a characterization with respect to the line set further, since there Grassmann embeddings are investigated and, hence, intersection properties of lines are of particular importance (see \cite{grassmann_embeddings_pasini_cardinali}).

Before we discuss the specific details and definitions in Section \ref{sec_def}, we will present the main results and their differences to \cite{maldeghem_splitcaley_by_intersect}.

\section{The Main Results}\label{sec_main}

Let $\scrL$ be a non-empty line set of $\PG(n,q)$. Consider the following properties:
\begin{description}
  \item[(Pt)] Every point of $\PG(n,q)$ is incident with either $0$ or $q+1$ elements of $\scrL$.
  \item[(Pl)] Every plane of $\PG(n,q)$ is incident with either $0, 1$ or $q+1$ elements of $\scrL$.
  \item[(Sd)] Every solid of $\PG(n,q)$ is incident with either $0, 1, q+1$ or $2q+1$ elements of $\scrL$.
  \item[(4d)] Every 4-dimensional subspace of $\PG(n,q)$ contains at most $q^3-q^2+4q$ elements of $\scrL$.
  \item[(To)] $|\scrL| \leq q^5+q^4+q^3+q^2+q+1$.
\end{description}

This is the main result:
\begin{satz}\label{main1}
  Let $\scrL$ be a non-empty line set of $\PG(n,q)$. The line set $\scrL$ is the line set of a split Cayley hexagon naturally embedded in $\PG(6,q)$ if and only if $\scrL$ satisfies (Pt), (Pl), (Sd), (4d), and (To).
\end{satz}

As we will see (4d) can be replaced by
\begin{description}
  \item[(Hp')] Every 5-dimensional subspace of $\PG(n,q)$ is incident with at most $q^4-q^3+3q^2+2q$ elements of $\scrL$.
\end{description}
This is a significant improvement of the main result of \cite{maldeghem_splitcaley_by_intersect} for $q > 2$: there the number of lines in a hyperplane is limited by 
\begin{description}
  \item[(Hp)] Every 5-dimensional subspace of $\PG(n,q)$ is incident with at most $q^3+3q^2+3q$ elements of $\scrL$.
\end{description}
Instead of (To) the main result of \cite{maldeghem_splitcaley_by_intersect} requires exactly $q^5+q^4+q^3+q^2+q+1$ lines in $\scrL$. Furthermore, the main result of \cite{maldeghem_splitcaley_by_intersect} requires $n=6$. The main result of \cite{maldeghem_splitcaley_by_intersect} is stronger in the sense that we can not replace condition (Sd) by
\begin{description}
  \item[(Sd')] Every solid of $\PG(n,q)$ is incident with at most $2q+1$ elements of $\scrL$,
\end{description}
since in contrast to \cite{maldeghem_splitcaley_by_intersect} we consider solids without meeting lines of $\scrL$ in the proof (if a solid $S$ of $\scrL$ contains two meeting lines of $\scrL$ and $\scrL$ satisfies (Pt), (Pl), and (Sd'), then $S$ contains $0$, $1$, $q+1$ or $2q+1$ lines of $\scrL$).

\begin{satz}\label{main2}
  Let $\scrL$ be a non-empty line set of $\PG(n,q)$. Then $\scrL$ is the line set of a naturally embedded split Cayley hexagon in $\PG(6,q)$ if and only if $\scrL$ satisfies (Pt), (Pl), (Sd), (Hp'), and (To).
\end{satz}

It is not possible to leave out or substantially weaken any of the conditions (Pt), (Pl), (Sd), (4d), or (To) in the main theorems, since in all these cases counter examples are (only) known by unpublished computer results\footnote{The author would like to thank the FinInG team (\url{http://cage.ugent.be/geometry/fining.php}) for providing an excellent GAP package for finite geometry. Without this package finding these examples would have been much harder.}.
Since these examples have no known short descriptions or interesting properties, we forgo any description of them. 
Maybe the only exception is (4d), where computer results suggest that in Theorem \ref{main1} the condition (4d) can be replaced by
\begin{description}
 \item[(6d)] Either $q > 3$ or $\dim(\langle \scrL \rangle) \geq 6$.
\end{description}

\section{Definitions, Notation and Terminology}\label{sec_def}

According to \cite[ch. 1]{maldeghem_generalizedpolygons} a \textit{geometry} (of rank $2$) is a triple $\Gamma = (\scrP, \scrL, \tI)$, where the point set $\scrP$ and the line set $\scrL$ are disjoint non-empty sets, and $\tI \subseteq \scrP \times \scrL$ is a symmetric relation (named \textit{incidence relation}). If two points $P, Q \in \scrP$ are incident with a common line $\ell$ of $\Gamma$, then $PQ := \ell$.

An \textit{$m$-gon} of $\scrL$ is a set $\{ P_1, \ldots, P_m\}$ of $m$ pairwise distinct points such that $P_iP_{i+1} \in \scrL$ for $i =1,2, \ldots, m$ (with $P_{m+1} := P_1$), and $P_iP_{i+1} \neq P_jP_{j+1}$ for $i \neq j$ and $j = 1,2, \ldots, m$. 
In this paper we will abuse the notation of sets in the case of $m$-gons by always ordering the vertices according to their adjacency as in the previous sentence, so for example if $\{A, B, C, D\}$ is a quadrangle, then $AB, BC, CD, DA \in \scrL$.
The geometry $\Gamma$ is called a \textit{generalized $m$-gon} if it satisfies the following two axioms:
\begin{enumerate}
  \item $\Gamma$ contains no $k$-gon for $2 \leq k < m$.
  \item For any two elements $x, y \in \scrP \cup \scrL$ there exists an $m$-gon $\{ P_1, \ldots, P_m\}$ with $x, y \in \{ P_i ~|~ 1 \leq i \leq m\} \cup \{ P_iP_{i+1} ~|~ 1 \leq i \leq m\}$.
\end{enumerate}
If every line of $\scrL$ is incident with exactly $s+1$ points of $\scrP$ and every point of $\scrP$ is incident with exactly $t+1$ lines of $\scrL$, then $\scrL$ has \textit{order} $(s, t)$. If $s,t >1$, then $\Gamma$ is called \textit{thick}.

Consider the projective space $\PG(n, q)$. An embedding of a geometry $\Gamma = (\scrP, \scrL, \tI)$ in $\PG(n, q)$ is an injective map of $\scrP$ in the point set of $\PG(n,q)$ inducing an injective map from $\scrL$ into the line set of $\PG(n, q)$. An embedding is a \textit{flat} embedding if all lines of $\scrL$ on a given point $P \in \scrP$ are coplanar in $\PG(n, q)$. An embedding is a \textit{full} embedding if $\Gamma$ has order $(q, t)$. Flat and full embeddings of generalized hexagons in $\PG(n, q)$ were partially classified by Thas and Van Maldeghem in \cite{thasmaldeghem_embeddings_split_caley}.

The split Cayley hexagon $\tH(q)$ can be defined by its \textit{natural embedding} in $\PG(6, q)$ in the following way (see \cite{MR1557095} or \cite[p. 73]{maldeghem_generalizedpolygons}): The point set of $\tH(q)$ is the point set of the parabolic quadric $\tQ(6, q)$ in $\PG(6, q)$ with $P = (x_0, x_1, x_2, x_3, x_4, x_5, x_6) \in \tQ(6, q)$ if and only if
\begin{align*}
  x_0x_4+x_1x_5+x_2x_6 = x_3^2.
\end{align*}
The lines of $\tH(q)$ are the lines of $\tQ(6, q)$ whose Grassmann coordinates $(p_{01}, \ldots, p_{65})$ satisfy the equations $p_{12}=p_{34}$, $p_{54}=p_{32}$, $p_{20}=p_{35}$, $p_{65}=p_{30}$, $p_{01}=p_{36}$, and $p_{46}=p_{31}$. The natural embedding for $\tH(q)$ in $\PG(6, q)$ is an example of a flat and full embedding. By (ii) of the main result of \cite{thasmaldeghem_embeddings_split_caley}, up to projectivity, the natural embedding of $\tH(q)$ is the only flatly and fully embedded thick generalized hexagon of order $(q,q)$ in $\PG(n, q)$ such that the point set of the generalized hexagon spans $\PG(6, q)$.

At some point in the proof we need a basic property of strongly regular graphs $\Gamma$ with parameters $(v, k, \lambda, \mu)$ (see \cite{brouwer1989distance} for details): If $\Gamma$ does not have the parameters $(4\mu+1, 2\mu, \mu-1, \mu)$, then the eigenvalues
\begin{align*}
 \frac{1}{2} \left( (\mu - \lambda) \pm \sqrt{(\mu-\lambda)^2 + 4(k-\mu)} \right)
\end{align*}
are integers.

Let $\scrL$ be a set of lines in $\PG(n,q)$ and $U$ a subspace of $\PG(n,q)$. Define
\begin{align*}
  \scrL_U := \{ \ell \in \scrL ~|~ \ell \I U \}.
\end{align*}
Let $P$ be a point, $E$ a plane and $S$ a solid of $\PG(n,q)$. We call
\begin{itemize}
  \item $P$ an $\alpha$-$U$-point if $|\scrL_U \cap \scrL_P| = \alpha$.
  \item $E$ an $\alpha$-$U$-plane if $|\scrL_U \cap \scrL_E| = \alpha$.
  \item $S$ an $\alpha$-$U$-solid if $|\scrL_U \cap \scrL_S| = \alpha$.
\end{itemize}
%

Define 
\begin{align*}
  \scrP := \{ P \in \PG(n,q) ~|~ \scrL_P \neq \emptyset \}.
\end{align*}
By (Pt) every point in $\scrP$ lies on $q+1$ lines of $\scrL$, and every line of $\scrL$ contains $q+1$ points of $\scrP$. Hence $|\scrP|=|\scrL|$.

For $P \in \scrP$ the subspace $\pi := \langle \scrL_P \rangle$ will be denoted by $\pi_P$.

\section{The standard embedding of $\tH(q)$ in $\PG(6,q)$ satisfies (Pt), (Pl), (Sd), (4d), and (To)}

Thas and Van Maldeghem show in \cite{maldeghem_splitcaley_by_intersect} that the standard embedding of $\tH(q)$ in $\PG(6,q)$ satisfies (Pt), (Pl), (Sd), (Hp), and (To). Since (Hp) implies (Hp'), it only has to be proven that the standard embedding satisfies (4d).

\begin{lemma}
  The standard embedding of $\tH(q)$ in $\PG(6, q)$ satisfies (4d).
\end{lemma}
\begin{proof}
  We have to show $|\scrL_U| \leq q^3-q^2+4q$ for every $4$-dimensional subspace $U$ of $\PG(6, q)$, and $\scrL$ the line set of the standard embedding of $\tH(q)$ in $\PG(6, q)$. According to \cite[p. 42]{hirschfeld1991general} the subspace $U$ can meet $\tQ(6,q)$ in four different ways:
  
  \begin{enumerate}
    \item $U \cap Q(6, q) = \tQ(4, q)$: In the standard embedding of $\tH(q)$ all (at least $3$) lines on a point are coplanar. A plane of $Q(4, q)$ is totally isotropic if it contains at least $3$ totally isotropic lines. The quadric $\tQ(4, q)$ contains no totally isotropic planes, hence every point of $\tQ(4, q)$ is incident with at most one line of $\scrL_U$. We have $|\tQ(4,q)| = (q^4-1)/(q-1)$ and every line of $\scrL_U$ contains $q+1$ points of $\tQ(4, q)$. So there are at most
    \begin{align*}
      \frac{q^4-1}{(q+1)(q-1)} = q^2+1 \leq q^3-q^2+4q
    \end{align*}
    lines in $\scrL_U$.
    
  \item $U \cap \tQ(6, q) = PQ^-(3, q)$: The only totally isotropic lines in $U \cap \tQ(6, q)$ are the lines on $P$. From $|Q^-(3,q)|=q^2+1$ it follows that there are exactly $q^2+1$ lines on $P$. Now
  \begin{align*}
    q^2+1 \leq q^3-q^2+4q
  \end{align*}
  yields the result.
  
  \item $U \cap \tQ(6, q) = P\tQp(3, q)$: There are $2(q+1)$ totally isotropic planes in $P\tQp(3, q)$, which meet in $P$. If one of those planes contains $q+1$ lines of $\scrL_U$, at least one of those lines contains $P$. Of course there are at most $q+1$ such lines and all such lines are coplanar. All those at most $q+1$ lines are in at most one additional $(q+1)$-$U$-plane, since a point in $\tQp(3, q)$ is incident with exactly two generators of $\tQp(3, q)$, hence 
  \begin{align*}
    |\scrL_U| \leq q+1 + (q+1)q = (q+1)^2 \leq q^3-q^2+4q.
  \end{align*}
  
  \item $U \cap \tQ(6, q) = \ell \tQ(2, q)$: There are at most $q+1$ totally isotropic planes in $\ell \tQ(2, q)$. Every plane in $\ell \tQ(2, q)$ contains at most $q+1$ lines of $\scrL_U$. As an upper bound follows
  \begin{align*}
    |\scrL_U| \leq (q+1)^2 \leq q^3-q^2+4q.
  \end{align*}
  \end{enumerate}
\end{proof}

\section{Proof of Theorem \ref{main1}}

In this section $\scrL$ denotes a line set of $\PG(n,q)$ satisfying (Pt), (Pl), (Sd), and (To). Also $\scrP$ denotes the set of points of $\PG(n, q)$ that are contained in $q+1$ lines of $\scrL$. The following two lemmas are proven in \cite{maldeghem_splitcaley_by_intersect}. Note that the proofs given in \cite{maldeghem_splitcaley_by_intersect} only use (Pt), (Pl), and (Sd) as stated in Section \ref{sec_main}. In particular, they do not rely on $n=6$:

\begin{lemma}[{\cite[Lemma 1]{maldeghem_splitcaley_by_intersect}}]
  Let $P$ be a point of $\scrP$. The $q+1$ lines in $\scrL_P$ are coplanar.
\end{lemma}

Hence, $\pi_P = \langle \scrL_P \rangle$ is a plane.

\begin{lemma}\label{lem_no3or4gons}
  The line set $\scrL$ contains no $3$- or $4$-gons.
\end{lemma}
\begin{proof}
 The first two paragraphs of the proof of Lemma 5 in \cite{maldeghem_splitcaley_by_intersect} prove the statement.
\end{proof}

Our goal is to show that $\scrL$ is a generalized hexagon, hence we have to show that $\scrL$ contains no pentagon. So we assume for the rest of this section that $\scrL$ contains a pentagon and investigate its properties.
\paragraph*{Assumption {\bf(AS)} for this section:} The line set $\scrL$ contains a pentagon $\scrF := \{A,B,C,D,E\} \subseteq \scrL$. Denote $\langle A, B, C, D, E \rangle$ by $U$.

\begin{lemma}\label{lem_dimU}
  The subspace $U$ has the properties $|\scrL_U| \geq 5q$ and $\dim(U) = 4$.
\end{lemma}
\begin{proof}
  The line set $\scrL$ contains no $3$-gons or $4$-gons. So we have 
  \begin{align*}
    \scrL_A \cap \scrL_C &= \scrL_A \cap \scrL_D = \{\},\\
    \scrL_A \cap \scrL_B &= \{ AB \}, \text{ and } \scrL_A \cap \scrL_E = \{ EA \}.
  \end{align*} 
  Hence,
  \begin{align*}
    | \scrL_A \setminus (\scrL_B \cup \scrL_C \cup \scrL_D \cup \scrL_E)| = | \scrL_A \setminus \{ AB, EA \}| = q-1.
  \end{align*}
  and equivalent statements for other points of $\scrF$. As $AB$, $BC$, $CD$, $DE$, and $EA$ lie in two of the sets,
  \begin{align*}
    |\scrL_U| \geq 5(q-1)+5 = 5q.
  \end{align*}
  Now $|\scrL_U| \geq 5q$, hence by (Sd) $\dim(U) \geq 4$. The subspace $U$ is the span of $5$ points, hence $\dim(U) \leq 4$.
\end{proof}

\begin{prop}\label{lem_nothreepoints}
  The only $(q+1)$-$U$-points on $EB$ are $B$ and $E$.
\end{prop}

\begin{proof}
  Suppose that there exists a $(q+1)$-$U$-point $P \in BE$ with $P \neq B, E$. Then the $(q+1)$-$U$-plane $\pi_P$ on $P$ meets the plane $\langle C, D, E \rangle$ in a point $Q$.
  
  Here $D \notin PQ$, since otherwise $\{A, P, D, E\}$ would be a quadrangle, and $C \notin PQ$, since otherwise $\{ A, P, C, B\}$ would be a quadrangle. Thus $PQ$ is not contained in $\langle C, D, E\rangle$ or $\langle B, C, D\rangle$.
  
  Hence the solid $\langle B, C, D, E \rangle$ contains the $(q+1)$-$U$-planes $\langle C, D, E\rangle$, $\langle B, C, D\rangle$, and the additional line $PQ$. Therefore $\langle B, C, D, E \rangle$ contains at least $2q+2$ lines, a contradiction to (Sd).
\end{proof}

\begin{prop}\label{lem_border4d}
  The subspace $U$ contains at least $q^3-q^2+4q+1$ lines of $\scrL$.
\end{prop}
\begin{proof}
  Set $\scrS := \left( (\scrL_B \setminus \{AB\}) \times (\scrL_E \setminus \{EA \}) \right) \setminus \{(BC, DE)\}$. 
  Consider a pair $(\ell, \ell') \in \scrS$ and the associated solid $S = \langle \ell, \ell' \rangle$.
  Notice that $\ell \cap \ell' = \emptyset$, since there do not exist quadrangles. By (Sd), the solid $S$ contain $q+1$ or $2q+1$ lines of $\scrL_U$. 
  
  Suppose that $S$ contains a line $s \in \bigcup_{P \in \scrF} \scrL_P$ with $s \neq \ell$ and $s \neq \ell'$. 
  W.l.o.g. we have the following three cases:
  \begin{enumerate}
    \item $s \in \scrL_E$: Hence $\pi_E = \langle s, \ell' \rangle \subseteq S$. 
    \item $s \in \scrL_A$: So $EA \subseteq \langle s, \ell' \rangle \subseteq S$. Hence $\pi_E = \langle EA, \ell' \rangle \subseteq S$.
    \item $s \in \scrL_D$: So $DE \subseteq \langle s, \ell' \rangle \subseteq S$. Hence $\pi_E = \langle DE, \ell' \rangle \subseteq S$.
  \end{enumerate}
  If $\pi_E \subseteq S$, then $A \in S$. Hence $\pi_B \subseteq \langle A, \ell \rangle \subseteq S$. Thus $U = \langle \pi_B, \pi_E \rangle \subseteq S$. This is a contradiction to Lemma \ref{lem_dimU}. Hence, such a line $s$ does not exist. Hence,
  \begin{align}
    \scrL_S \cap \bigcup_{P \in \scrF} \scrL_P = \{ \ell, \ell'\}.\label{eq_cut_empty1}
  \end{align}
  
  Now consider pairs $(\ell_1, \ell_1') \in \scrS$ and $(\ell_2, \ell_2') \in \scrS$ of disjoint lines with associated solids $S_1 := \langle \ell_1, \ell_1' \rangle$ and $S_2 := \langle \ell_2, \ell_2' \rangle$. We want to show that $S_1 \cap S_2$ contains no line of $\scrL \setminus (\scrL_B \cup \scrL_E)$: The plane $S_1 \cap S_2$ contains $BE$. Hence every line $s$ of $\scrL \setminus (\scrL_B \cup \scrL_E)$ in $S_1 \cap S_2$ has to meet $BE \subseteq \pi_A$ in a point $P$. The solids $S_1$ and $S_2$ contain no line of $\scrL_A$ by \eqref{eq_cut_empty1}, so $\pi_A \nsubseteq S_1 \cap S_2$. In particular $s \notin \scrL_A$. Hence $P$ is contained in a line $t$ of $\scrL_A$ with $t \neq s$. Thus $s, t \in \scrL_P$, and $P$ is a $(q+1)$-$U$-point. This is a contradiction to Proposition \ref{lem_nothreepoints}. Hence such a line $s$ does not exist. Hence together with \eqref{eq_cut_empty1} and $|\scrS| = q^2-1$, $U$ contains at least
  \begin{align*}
    |\scrL_U| &\geq \sum_{(\ell, \ell') \in \scrS} |\scrL_{\langle \ell, \ell' \rangle}| + \left| \bigcup_{P \in \scrF} \scrL_P\right| \\
    &\geq (q^2-1)(q-1) + 5q = q^3 - q^2 + 4q + 1
  \end{align*}
  lines.
\end{proof}

\begin{proof}[Proof of Theorem \ref{main1}]
If $\scrL$ contains a pentagon, hence {\bf(AS)} is true, then $\scrL$ does contain a subspace $U$ with at least $q^3-q^2+4q+1$ lines of $\scrL$ by Proposition \ref{lem_border4d}. Hence, if $\scrL$ also satisfies (4d), then $\scrL$ contains no pentagon. In this case (as in \cite{maldeghem_splitcaley_by_intersect}) a standard counting argument shows that $\scrL$ is the line set of a generalized hexagon of order $q$ flatly and fully embedded in $\PG(n,q)$. Thus (ii) of the main result of \cite{thasmaldeghem_embeddings_split_caley} implies that $\scrL$ is isomorphic to the line set of the standard embedding of $\tH(q)$ in $\PG(6,q)$.
\end{proof}

\section{Proof of Theorem \ref{main2}}

Now let $\scrL$ be a line set that satisfies (Pt), (Pl), (Sd), and (To), but not necessarily (4d). In this section we want to show that (4d) is implied by (Hp') to compare (4d) to condition (Hp) of \cite{maldeghem_splitcaley_by_intersect}. For this we will show that a pentagon in $\scrL$ implies approximately $q^4$ lines of $\scrL$ in a hyperplane of $\PG(n,q)$.

\begin{lemma}\label{lem_expand1}
 Let $M$ be s subspace of $\PG(n,q)$. Let $\ell \in \scrL \setminus \scrL_M$ be a line that meets one line $s \in \scrL_M$. 
 \begin{enumerate}[(a)]
  \item Then $s$ is the only line of $\scrL_M$ meeting $\ell$.
  \item If $\ell^M \in \scrL_M$ meets $s$ exactly in a point, then $\scrL_{\langle \ell^M, \ell \rangle} = \scrL_{\langle \ell^M, s \rangle} \cup \scrL_{\langle s, \ell \rangle}$.
  \item If $\ell^M_1, \ell^M_2 \in \scrL_M$ meet $s$ exactly in points $P$ resp. $Q$, then
    \begin{align*}
      \scrL_{\langle \ell^M_1, \ell \rangle} \cap \scrL_{\langle \ell^M_2, \ell \rangle} = 
      \begin{cases}
       \scrL_{\langle \ell^M_1, \ell \rangle} & \text{ if } P = Q,\\
       \scrL_{\langle s, \ell \rangle} & \text{ if } P \neq Q.
      \end{cases}
    \end{align*}
 \item If $\ell^M_1 \in \scrL_M$ meets $s$ exactly in a point $P$ and $\ell^M_2 \in \scrL_M$ is disjoint from $s$, then 
 \begin{align*}
  \scrL_{\langle \ell^M_1, \ell \rangle} \cap \scrL_{\langle \ell^M_2, \ell \rangle} = \{ \ell \}.
 \end{align*}
 \end{enumerate}
\end{lemma}
\begin{proof}
 \begin{enumerate}[(a)]
  \item The point $\ell \cap s$ is a $(q+1)$-point $P$ and $\pi_P \cap M = s$.
  \item Each of the planes $\langle \ell^M, s \rangle$ and $\langle s, \ell \rangle$ contains two lines of $\scrL$, hence they are both $(q+1)$-planes by (Pl). By (Sd) the $2q+1$ line of $\scrL_{\langle \ell^M, s \rangle} \cup \scrL_{\langle s, \ell \rangle}$ are all lines of $\scrL_{\langle \ell^M, \ell \rangle}$. The assertion follows.
  \item Since $\scrL_{\langle \ell^M_1, s \rangle} = \scrL_P$ and $\scrL_{\langle \ell^M_2, s \rangle} = \scrL_Q$, (b) shows the assertion.
  \item Let $\ell' \in \scrL_{\langle \ell, \ell^M_1 \rangle} \cap \scrL_{\langle \ell, \ell^M_2 \rangle}$.
As in (b) $\scrL_{\langle \ell, \ell^M_1 \rangle} = \scrL_{\langle \ell, s\rangle} \cup \scrL_{\langle s, \ell^M_1 \rangle}$, hence $\ell' \in \scrL_{\langle \ell, s\rangle}$ or $\ell' \in \scrL_{\langle s, \ell^M_1 \rangle}$.
\begin{enumerate}[1.]
 \item If $\ell' \in \scrL_{\langle \ell, s\rangle}$, then $\ell' \in \scrL_{\langle \ell, s\rangle} \cap \scrL_{\langle \ell, \ell^M_2 \rangle} = \{ \ell\}$ as $s$ does not lie in the plane $\langle \ell, \ell^M_2 \rangle \cap M$, since $s$ is disjoint from $\ell_2^M$.
 \item If $\ell' \in \scrL_{\langle s, \ell^M_1 \rangle}$, then $\ell' \cap s, \ell \cap s \in \langle \ell, \ell^M_2 \rangle \cap M$. Hence $s$ is contained in $\langle \ell, \ell^M_2 \rangle \cap M$, contradicting $s \nsubseteq \langle \ell, \ell_2^M \rangle \cap M$.
\end{enumerate}
Hence, $\ell = \ell'$. Hence, $\scrL_{\langle \ell, \ell^M_1 \rangle} \cap \scrL_{\langle \ell, \ell^M_2 \rangle} = \{ \ell \}$.
 \end{enumerate}
\end{proof}

\begin{lemma}\label{lem_expand2}
 Let $M$ be s subspace of $\PG(n,q)$. Let $\ell \in \scrL \setminus \scrL_M$. Let
 \begin{align*}
  \scrS := \{ \ell^M \in \scrL_M ~|~ \scrL_{\langle \ell^M, \ell \rangle} \cap \scrL_M = \{ \ell^M \} \text{ and } \ell^M \cap \ell = \emptyset \}.
 \end{align*}
  If $\ell^M_1 \in \scrS$, then $\scrL_{\langle \ell^M_1, \ell \rangle}$ contains at least $q$ lines not in
  \begin{align*}
    \bigcup_{\ell^M_2 \in \scrS \setminus \{ \ell^M_1 \}} \scrL_{\langle \ell^M_2, \ell \rangle}.
  \end{align*}
\end{lemma}
\begin{proof}
Let $\ell^M_1, \ell^M_2 \in \scrS$ with $\ell^M_1 \neq \ell^M_2$. Set $\pi := \langle \ell, \ell^M_1 \rangle \cap \langle \ell, \ell^M_2 \rangle$. By definition of $\scrS$, $\dim(\pi) \leq 2$. Since $\ell \in \pi$, either $\pi = \ell$, or $\pi$ is a plane with $\scrL_\pi \in \{1,q+1\}$ by (Pl).

If $\dim(\pi) \leq 1$ or $\pi$ is a $1$-plane, then clearly $\scrL_{\langle \ell, \ell^M_1 \rangle} \cap \scrL_{\langle \ell, \ell^M_2 \rangle} = \{ \ell \}$.
Hence, if only this case occurs, then the assertion is obviously true.
 
If $\pi$ is a $(q+1)$-plane, then $\pi$ contains no line of $\scrL_M$, since $\pi \cap M \neq \ell^M_1, \ell^M_2$. Hence, the line $t \in \scrL_\pi$ that meets $\ell^M_1$ is not contained in $M$. Hence, $\langle t, \ell^M_1 \rangle$ is a $(q+1)$-plane not in $M$.
 Thus $\langle \ell , \ell^M_1 \rangle$ is a $(2q+1)$-solid that consists of the two $(q+1)$-planes $\langle t, \ell^M_1 \rangle$ and $\pi$. 
 Any line $\ell' \in \scrL_{\langle \ell, \ell^M_1 \rangle}$ with $\ell' \nsubseteq \pi$ is in $\scrL_{\langle t, \ell^M_1 \rangle}$, hence $\langle \ell, \ell' \rangle = \langle \ell, \ell^M_1 \rangle$. 
 
 Suppose that there exists a line $\ell^M_3 \in \scrS$ with $\ell' \subseteq \langle \ell, \ell^M_3 \rangle$. 
 Then $\langle \ell, \ell' \rangle = \langle \ell, \ell^M_1 \rangle$ implies $\ell^M_3 \subseteq \langle \ell, \ell^M_1 \rangle \cap M = \pi \cap M$.
 This contradicts $\scrL_{\pi \cap M} = \{ \ell^M_1 \}$. Hence, in the case that $\pi$ is a $(q+1)$-plane we find
 \begin{align*}
  \left|\scrL_{\langle \ell, \ell^M_1\rangle} \setminus \bigcup_{\ell^M_1 \neq \ell^M_2 \in \scrS} \scrL_{\langle \ell, \ell^M_2\rangle }\right| = \left| \scrL_{\langle t, \ell^M_1 \rangle} \setminus \{ \ell \} \right| = q.
 \end{align*}
\end{proof}

\begin{prop}\label{lem_expand}
  Let $M$ be a subspace of $\PG(n,q)$. Let $\ell \in \scrL$ be a line that meets $M$ in exactly one point.
  \begin{enumerate}[(a)]
  \item If $\ell$ meets no line of $\scrL_M$, then
  \begin{align*}
    |\scrL| \geq q|\scrL_M|+1.
  \end{align*}
  \item If $\ell$ meets a line $s$ of $\scrL_M$ with $\alpha$ $(q+1)$-$M$-points, then
  \begin{align*}
    |\scrL| \geq q|\scrL_M| - \alpha q^2 + \alpha q + 1 \geq q|\scrL_M| - q^3+q^2+1.
  \end{align*}
  \end{enumerate}
\end{prop}
\begin{proof}
If $s$ exists, then set
\begin{align*}
 &\scrS_1 := \{ \ell^M \in \scrL_M ~|~ \scrL_{\langle \ell^M, \ell \rangle} \cap \scrL_M = \{ \ell^M \}  \text{ and } \ell^M \cap \ell = \emptyset \},\\
 &\scrS_2 := \{ \ell^M \in \scrL_M ~|~ \scrL_{\langle \ell^M, \ell \rangle} \cap \scrL_M = \scrL_{\langle s, \ell^M \rangle}  \}.
\end{align*}
By Lemma \ref{lem_expand1} (b), all lines which meet $s$ in exactly a point are in $\scrS_2$. Conversely, if $\ell^M \in \scrS_2$, then $\langle s, \ell^M \rangle$ is a plane and $\ell^M$ meets $s$ in exactly a point. Hence, we can apply Lemma \ref{lem_expand1}. Furthermore, for the same reasons $\scrL_M = \scrS_1 \cup \scrS_2 \cup \{ s \}$.

If $s$ does not exists, then set $\scrS_1 := \scrL_M$ and $\scrS_2 := \emptyset$. By Lemma \ref{lem_expand2},
\begin{align*}
 &\left| \bigcup_{\ell^M \in \scrS_1} \scrL_{\langle \ell, \ell^M \rangle} \right| \geq q|\scrS_1| + 1.
\intertext{This shows (a). Hence from now on we suppose that $s$ exists. By Lemma \ref{lem_expand1} (c),}
 &\left| \bigcup_{\ell^M \in \scrS_2} \scrL_{\langle \ell, \ell^M \rangle} \right| = |\scrS_2| + q + 1.
\intertext{By Lemma \ref{lem_expand1} (d),}
 &\left( \bigcup_{\ell^M \in \scrS_1} \scrL_{\langle \ell, \ell^M \rangle} \right) \cap \left( \bigcup_{\ell^M \in \scrS_2} \scrL_{\langle \ell, \ell^M \rangle} \right)  = \{ \ell \}.
\intertext{Hence,}
 &\left| \bigcup_{\ell^M \in \scrS_1 \cup \scrS_2} \scrL_{\langle \ell, \ell^M \rangle} \right| \geq q|\scrS_1| + |\scrS_2| + q + 1.
\end{align*}

By definition, $|\scrS_2| = q\alpha$, and $|\scrS_1| = |\scrL_M \setminus (\scrS_2 \cup \{s \})| = |\scrL_M| - \alpha q - 1$. Hence,
  \begin{align*}
    |\scrL| &\geq q|\scrS_1| + |\scrS_2| + q + 1\\
    &= q(|\scrL_M| - \alpha q - 1) + q\alpha + q + 1 \\
    &= q|\scrL_M| - \alpha q(q-1) + 1.
  \end{align*}
  As $\alpha \leq q$ we find
  \begin{align*}
    |\scrL| \geq q|\scrL_M| - q^3 + q^2 + 1.
  \end{align*}
\end{proof}.

\begin{lemma}\label{lem_ineverypentagon}\label{lem_twoineverypentagon}
  Let $U$ be a $4$-dimensional subspace of $\PG(n, q)$ which contains a pentagon.
  \begin{enumerate}[(a)]
    \item If $P$ is a $(q+1)$-$U$-point and $V$ is a vertex of a $5$-gon contained in $U$ such that $P \neq V$ and $PV \in \scrL$, then there exists a pentagon of $U$ containing $P$ and $V$.\label{lem_twoineverypentagon_a}
    \item Every $(q+1)$-$U$-point $P$ is the vertex of a $5$-gon of $U$.\label{lem_twoineverypentagon_b}
    \item Suppose $P$, $Q$, $R$ are $(q+1)$-$U$-points with $P \neq Q, R$ and $Q, R \in \pi_P$. Suppose furthermore that $R = Q$ or $R \notin PQ$. Then there exists a pentagon of $U$ such that $P$, $Q$, $R$ are vertices of it.
  \end{enumerate}
\end{lemma}
\begin{proof}
  \begin{enumerate}[(a)]
  \item Let $\{ V, W, X, Y, Z\}$ be the vertices of a pentagon of $U$. The assertion is trivial if $P=Z$, so suppose that $P \neq Z$. The plane $\pi_P$ shares a point $Q$ with $\pi_Y$. As $V$, $P$, $Q$, $Y$, $Z$ are non-collinear, then $\{ V, P, Q, Y, Z\}$  is a pentagon, since $\scrL$ does not contain $3$-gons or $4$-gons.
  \item The planes $\pi_P$ and $\pi_A$ share a point $Q$. By \eqref{lem_twoineverypentagon_a}, the point $Q$ lies in a pentagon of $U$, and thus, again by \eqref{lem_twoineverypentagon_a}, $P$ lies in a pentagon of $U$.
  \item In view of \eqref{lem_twoineverypentagon_b} we may assume that $P=A$ and in view of \eqref{lem_twoineverypentagon_a} we may assume that $Q = B$. If $R=Q$, the statement is trivial. Otherwise, consider a common point $T$ of $\pi_R$ and $\pi_C$. As $P$, $Q$, $C$ are non-collinear, then $\{P, Q, C, T, R\}$ is a pentagon.
  \end{enumerate}
\end{proof}

Let $M$ be a $4$-dimensional subspace of $\PG(n,q)$.

\begin{lemma}\label{kor_bound_qp1_points_plane}
  Let $P$ be a $(q+1)$-$M$-point. Then $\pi_P$ contains at most $q+2$ $(q+1)$-$M$-points.
\end{lemma}
\begin{proof}
  For a $(q+1)$-$M$-point $X \neq P$ of $\pi_P$, the planes $\pi_P$ and $\pi_X$ share the line $PX$ and thus span a solid of $M$. Different such points $X$ give rise to different solids by (Sd). As $\pi_P$ lies on only $q+1$ solids of $M$, the assertion follows.
\end{proof}

\begin{kor}\label{kor_cases_q2plane}
  Let $P$ be a $(q+1)$-$M$-point. Then one of the following cases occurs:
  \begin{enumerate}[(a)]
    \item Every line on $P$ in $\pi_P$ contains exactly one $(q+1)$-$M$-point other than $P$.\label{kor_cases_q2plane_a}
    \item There exists one line on $P$ in $\pi_P$ with no $(q+1)$-$M$-point other than $P$.\label{kor_cases_q2plane_b}
  \end{enumerate}
\end{kor}
\begin{proof}
  If the first assertion is false, then the second holds by Lemma \ref{kor_bound_qp1_points_plane}.
\end{proof}

\begin{prop}\label{kor_bound_qp1_points}
  If $M$ contains at least one $(q+1)$-$M$-point, then one of the following cases occurs:
  \begin{enumerate}[(a)]
   \item one line in $\scrL_M$ contains exactly one $(q+1)$-$M$-point,
   \item $q=2$, $|\scrL_M| \geq 15$, and there exists a line in $\scrL_M$ with exactly two $(q+1)$-$M$-points.
  \end{enumerate}
\end{prop}
\begin{proof}
By Corollary \ref{kor_cases_q2plane}, if one line of $\scrL_M$ does not contain exactly none or two $(q+1)$-$M$-points, then one line of $\scrL_M$ contains exactly one $(q+1)$-$M$-point.

Hence consider the case that each line of $\scrL_M$ with $(q+1)$-$M$-points contains exactly none or two $(q+1)$-$M$-points. Let 
  \begin{align*}
   \scrL_M' := \{ \ell \in \scrL_M ~|~ \ell \text{ contains a $(q+1)$-$M$-point} \},
  \end{align*}
  and let $\scrP_M'$ be the set of $(q+1)$-$M$-points. Consider the strongly regular graph $\scrG := (\scrP_M',\scrL_M', \in)$ with $k=q+1$, $\lambda=0$, $\mu=1$, girth $5$ (by Lemma \ref{lem_no3or4gons}) and diameter $2$ (by Lemma \ref{lem_twoineverypentagon} (c)). Counting the vertices of $\scrG$ by their distance to a fixed vertex shows
  \begin{align*}  
   |\scrP_M'| = 1 + (q+1) + (q+1)q = 2+2q+q^2.
  \end{align*}
  Hence $\scrG$ has the parameters $(v, k, \lambda, \mu) = (2+2q+q^2, q+1, 0, 1)$ and $|\scrL_M'| = \frac{(q+1)|\scrP_M'|}{2} = \frac{1}{2} q^3 + \frac{3}{2} q^2 + 2q + 1$. Here $k = q+1 \neq 2 = 2\mu$, so $\scrG$ is not a conference graph. Thus the eigenvalues
  \begin{align*}
    \frac{1}{2}\left( -1 \pm \sqrt{1+4q} \right)
  \end{align*}
  of the graph have to be integers. Hence $1+4q$ has to be an odd square. This condition restricts the possible values of $q$ to $u(u+1)$ for a natural number $u$. Obviously, $q=2$ is the only prime power satisfying this condition.
\end{proof}

\begin{prop}\label{kor_scrl_pentagon_kor1}
  If $\scrL$ contains a $5$-gon, then there exists a $5$-dimensional subspace $H$ with
  \begin{align*}
    |\scrL_H| \geq q^4-q^3+3q^2+2q+1.
  \end{align*}
\end{prop}
\begin{proof}
  In the first case of Proposition \ref{kor_bound_qp1_points} there exists a line $s \in \scrL_U$ with at most one $(q+1)$-$U$-point. Hence $\alpha \leq 1$ in Proposition \ref{lem_expand}. Thus by Proposition \ref{lem_border4d}
  \begin{align*}
    |\scrL_H| \geq q|\scrL_U| -q^2+q+1 \geq q^4-q^3+3q^2+2q+1.
  \end{align*}
  In the second case of Proposition \ref{kor_bound_qp1_points} $q=2$,  and there exists a line $s \in \scrL_U$ with at most two $(q+1)$-$U$-points and $|\scrL_U| \geq 15$. Hence $\alpha \leq 2$ in Proposition \ref{lem_expand} yields
  \begin{align*}
    |\scrL_H| \geq q|\scrL_U| - 2q^2+2q+1 = 27 > 25 = q^4-q^3+3q^2+2q+1.
  \end{align*}
\end{proof}

\begin{proof}[Proof of Theorem \ref{main2}]
 By Proposition \ref{kor_scrl_pentagon_kor1} the existence of a pentagon in $\scrL$ implies the existence of a hyperplane that does not satisfy (Hp'). Thus (ii) of the main result of \cite{thasmaldeghem_embeddings_split_caley} implies that $\scrL$ is isomorphic to the line set of the standard embedding of $\tH(q)$ in $\PG(6,q)$. This completes the proof of Theorem \ref{main2}.
\end{proof}

Finally, we want to mention the following observation.

\begin{satz}\label{kor_scrl_pentagon_kor1a}
  If $\scrL$ satisfies (Pt), (Pl), (Sd), and (To), then $\scrL$ is contained in a $6$-dimensional subspace of $\PG(n,q)$.
\end{satz}
\begin{proof}
  If $\scrL$ contains no $5$-gon, then $\scrL$ is the line set of a standard embedding of the split Cayley hexagon by the main result of \cite{thasmaldeghem_embeddings_split_caley}. 
  Hence consider the remaining case that $\scrL$ contains a $5$-gon. Suppose that $\scrL$ is not contained in a $6$-dimensional subspace. If $\scrL$ is connected, applying Proposition \ref{lem_expand} twice with $\alpha=q$ contradicts (To). More precisely, let $U$ be a $4$-dimensional subspace with a pentagon, $V$ a $5$-dimensional subspace with $U \subseteq V$, and $W$ a $6$-dimensional subspace with $V \subseteq W$. By Proposition \ref{kor_scrl_pentagon_kor1}, $\scrL_V \geq q^4-q^3+3q^2+2q+1$. Hence
  \begin{align*}
    |\scrL| &\geq q|\scrL_W| - q^3+q^2+1\\
      &\geq q(q|\scrL_V|-q^2+q+1) - q^3+q^2+1\\
      &\geq q^6-q^5+2q^2+2q^3+2q^2+q+1 \\
      &> q^5+q^4+q^3+q^2+q+1 \geq |\scrL|.
  \end{align*}
  If $\scrL$ is not connected, let $\scrL_1$ be a connected component of $\scrL$ and $\scrL_2 := \scrL \setminus \scrL_1$. By requirement $\dim(\langle \scrL_1, \scrL_2 \rangle) \geq 7$. If $\dim(\langle \scrL_1 \rangle) \geq 7$, then the previous calculation is applicable. If $\dim(\langle \scrL_1 \rangle) < 7$, then $\scrL_2 \nsubseteq \langle \scrL_1 \rangle$. 
  So we find a line $\ell$ in $\scrL_2$ such that $\ell$ meets $\langle \scrL_1 \rangle$ in a point or $\dim(\langle \scrL_1 \rangle \cap \langle \scrL_2 \rangle) =0$. In the first case, we can repeat the same calculation as before. In the second case, $\dim(\scrL_2) \geq 4$ by Lemma \ref{lem_dimU}. Hence, we find lines $\ell_1, \ell_2 \in \scrL_2$ such that we can repeat the estimate of Proposition \ref{lem_expand} with $\alpha = q$ at least two additional times: first for $\langle \scrL_1, P \rangle$ and $\ell_1$, where $P$ is a point of $\ell_1$, and then for $\langle \scrL_2, \ell_1 \rangle$ and $\ell_2$. Again, this contradicts (To).
\end{proof}

\section*{Acknowledgement}

The author thanks several anonymous referees and Klaus Metsch for their helpful advice on the presentation of the results.

\bibliographystyle{elsarticle-num}
\bibliography{literatur}

\begin{thebibliography}{10}
\expandafter\ifx\csname url\endcsname\relax
  \def\url#1{\texttt{#1}}\fi
\expandafter\ifx\csname urlprefix\endcsname\relax\def\urlprefix{URL }\fi
\expandafter\ifx\csname href\endcsname\relax
  \def\href#1#2{#2} \def\path#1{#1}\fi

\bibitem{MR1557095}
J.~Tits, \href{http://www.numdam.org/item?id=PMIHES_1959__2__13_0}{Sur la
  trialit\'e et certains groupes qui s'en d\'eduisent}, Inst. Hautes \'Etudes
  Sci. Publ. Math.~(2) (1959) 13--60.
\newline\urlprefix\url{http://www.numdam.org/item?id=PMIHES_1959__2__13_0}

\bibitem{MR2964273}
J.~Bamberg, N.~Durante, \href{http://dx.doi.org/10.1090/conm/579/11514}{Low
  dimensional models of the finite split {C}ayley hexagon}, in: Theory and
  applications of finite fields, Vol. 579 of Contemp. Math., Amer. Math. Soc.,
  Providence, RI, 2012, pp. 1--19.
\newblock \href {http://dx.doi.org/10.1090/conm/579/11514}
  {\path{doi:10.1090/conm/579/11514}}.
\newline\urlprefix\url{http://dx.doi.org/10.1090/conm/579/11514}

\bibitem{MR549937}
P.~J. Cameron, W.~M. Kantor,
  \href{http://dx.doi.org/10.1016/0021-8693(79)90090-5}{{$2$}-transitive and
  antiflag transitive collineation groups of finite projective spaces}, J.
  Algebra 60~(2) (1979) 384--422.
\newblock \href {http://dx.doi.org/10.1016/0021-8693(79)90090-5}
  {\path{doi:10.1016/0021-8693(79)90090-5}}.
\newline\urlprefix\url{http://dx.doi.org/10.1016/0021-8693(79)90090-5}

\bibitem{hirschfeld1991general}
J.~Hirschfeld, J.~Thas, General Galois geometries, Oxford mathematical
  monographs, Clarendon Press, 1991.

\bibitem{MR2663559}
S.~De~Winter, J.~Schillewaert,
  \href{http://dx.doi.org/10.1007/s00493-010-2441-2}{Characterizations of
  finite classical polar spaces by intersection numbers with hyperplanes and
  spaces of codimension 2}, Combinatorica 30~(1) (2010) 25--45.
\newblock \href {http://dx.doi.org/10.1007/s00493-010-2441-2}
  {\path{doi:10.1007/s00493-010-2441-2}}.
\newline\urlprefix\url{http://dx.doi.org/10.1007/s00493-010-2441-2}

\bibitem{MR2927612}
J.~Schillewaert, J.~A. Thas, H.~Van~Maldeghem,
  \href{http://dx.doi.org/10.1007/s00026-012-0136-7}{A characterization of the
  finite {V}eronesean by intersection properties}, Ann. Comb. 16~(2) (2012)
  331--348.
\newblock \href {http://dx.doi.org/10.1007/s00026-012-0136-7}
  {\path{doi:10.1007/s00026-012-0136-7}}.
\newline\urlprefix\url{http://dx.doi.org/10.1007/s00026-012-0136-7}

\bibitem{MR2128336}
J.~A. Thas, H.~Van~Maldeghem,
  \href{http://dx.doi.org/10.1007/s10623-004-4860-9}{Some characterizations of
  finite {H}ermitian {V}eroneseans}, Des. Codes Cryptogr. 34~(2-3) (2005)
  283--293.
\newblock \href {http://dx.doi.org/10.1007/s10623-004-4860-9}
  {\path{doi:10.1007/s10623-004-4860-9}}.
\newline\urlprefix\url{http://dx.doi.org/10.1007/s10623-004-4860-9}

\bibitem{MR2860661}
J.~A. Thas, H.~Van~Maldeghem,
  \href{http://dx.doi.org/10.1007/s00022-011-0091-1}{Characterizations of
  {V}eronese and {S}egre varieties}, J. Geom. 101~(1-2) (2011) 211--222.
\newblock \href {http://dx.doi.org/10.1007/s00022-011-0091-1}
  {\path{doi:10.1007/s00022-011-0091-1}}.
\newline\urlprefix\url{http://dx.doi.org/10.1007/s00022-011-0091-1}

\bibitem{MR2515274}
J.~A. Thas, H.~Van~Maldeghem, 1-polarized pseudo-hexagons, Innov. Incidence
  Geom. 6/7 (2007/08) 307--325.

\bibitem{MR1633172}
H.~Cuypers, A.~I. Steinbach,
  \href{http://dx.doi.org/10.1515/jgth.1998.015}{Weak embeddings of generalized
  hexagons and groups of type {$G_2$}}, J. Group Theory 1~(3) (1998) 225--236.
\newblock \href {http://dx.doi.org/10.1515/jgth.1998.015}
  {\path{doi:10.1515/jgth.1998.015}}.
\newline\urlprefix\url{http://dx.doi.org/10.1515/jgth.1998.015}

\bibitem{maldeghem_splitcaley_by_intersect}
J.~A. Thas, H.~Van~Maldeghem,
  \href{http://dx.doi.org/10.1016/j.ejc.2007.06.008}{A characterization of the
  natural embedding of the split {C}ayley hexagon {$H(q)$} in {${\rm PG}(6,q)$}
  by intersection numbers}, European J. Combin. 29~(6) (2008) 1502--1506.
\newblock \href {http://dx.doi.org/10.1016/j.ejc.2007.06.008}
  {\path{doi:10.1016/j.ejc.2007.06.008}}.
\newline\urlprefix\url{http://dx.doi.org/10.1016/j.ejc.2007.06.008}

\bibitem{grassmann_embeddings_pasini_cardinali}
A.~Pasini, I.~Cardinali, Groups and Geometries (ed. {N.S.N. Sastry}), Springer
  Proceedings in Mathematics, Springer-Verlag, in press, Ch. Embeddings of
  Line-grassmannians of Polar Spaces in {G}rassmann Varieties.

\bibitem{maldeghem_generalizedpolygons}
H.~{van Maldeghem}, Generalized Polygons, Birkhäuser Basel, 1998.

\bibitem{thasmaldeghem_embeddings_split_caley}
J.~A. Thas, H.~{Van Maldeghem}, Flat lax and weak embeddings of finite
  generalized hexagons, European J. Combin. 19 (1998) 733 -- 751.

\bibitem{brouwer1989distance}
A.~Brouwer, A.~Cohen, A.~Neumaier, Distance-regular graphs, Ergebnisse der
  Mathematik und ihrer Grenzgebiete, Springer, 1989.

\end{thebibliography}

\appendix

\end{document}